\documentclass{article}

\usepackage{amsmath}
\usepackage{amsthm}
\usepackage{amssymb}
\usepackage{tikz}
\usepackage{soul}

\usetikzlibrary{arrows,decorations.pathreplacing}

\theoremstyle{plain}
\newtheorem{theorem}{Theorem}
\newtheorem{lemma}[theorem]{Lemma}

\newtheorem{corollary}[theorem]{Corollary}
\newtheorem{question}{Question}

\theoremstyle{definition}
\newtheorem{example}[theorem]{Example}
\newtheorem{definition}[theorem]{Definition}
\newtheorem{construction}[theorem]{Construction}

\newcommand{\defn}[1]{\emph{#1}}

\newcommand{\CBC}{{\rm\textnormal{-}ECBC}}
\newcommand{\dash}{{\rm\textnormal{-}}}

\title{On Erasure Combinatorial Batch Codes}

\author{
JiYoon Jung\thanks{Marshall University, 1 John Marshall Drive, Huntington, WV USA, 25755} \and
Carl Mummert\footnotemark[1] \and
Elizabeth Niese\footnotemark[1] \and
Michael W. Schroeder\footnotemark[1]~\thanks{Corresponding author: {\tt schroederm@marshall.edu}}}

\date{\vspace{-5ex}}

\begin{document}
\maketitle

\begin{abstract}
Combinatorial batch codes were defined by Paterson, Stinson, and Wei as purely combinatorial versions of the batch codes introduced by Ishai, Kushilevitz, Ostrovsky, and Sahai. 
There are $n$ items and $m$ servers, each of which stores a subset of the items. 
A batch code is an arrangement for storing items on servers so that, for prescribed integers $k$ and $t$, any $k$ items can be retrieved by reading at most $t$ items from each server. 
Silberstein defined an erasure batch code (with redundancy $r$) as a batch code in which any $k$ items can be retrieved by reading at most $t$ items from each server, while any $r$ servers are unavailable (failed).

In this paper, we investigate erasure batch codes with $t=1$ (each server can read at most one item) in a combinatorial manner.
We determine the optimal (minimum) total storage of an erasure batch code for several ranges of parameters.
Additionally, we relate optimal erasure batch codes to maximum packings.
We also identify a necessary lower bound for the total storage of an erasure batch code, and we relate parameters for which this trivial lower bound is achieved to the existence of graphs with appropriate girth.

\medskip
\noindent
{\bf\begin{tabular}{@{}l@{}}
AMS Classification: \quad 05B05, 05B30 \\
Keywords: \quad \begin{tabular}[t]{@{}l@{}} batch codes, dual codes, erasure batch codes, \\ block designs, dual designs\end{tabular}
\end{tabular}}

\end{abstract}

\section{Introduction}

We study a class of combinatorial objects called \emph{erasure combinatorial batch codes}. 
These are motivated by a data retrieval problem in which a collection of items (such as files) are stored, with possible  duplication, on a collection of servers. After the items are stored, a demand will be made for at most $k$ distinct items, where $k$ is fixed in advance. 
The goal is to store as few total copies of the items as possible, while still being able to retrieve any $q$-subset ($q\leq k$) of the items without taking too many from any one server. 
The study of such problems was initiated by Ishai, Kushilevitz, Ostrovsky, and Sahai~\cite{ishai}, who proposed a class of \emph{batch codes} that provide solutions to the following particular data retrieval problem. 
%The study of such problems was initiated by Ishai, Kushilevitz, Ostrovsky, and Sahai~\cite{ishai}, and was investigated in a purely combinatorial manner by who proposed a class of \emph{batch codes} that provide solutions to the a particular version of this retrieval problem.
%Paterson, Stinson, and Wei posed this questions in a purely combinatorial manner: 

\begin{question}[Ishai~et.~al.~\cite{ishai}]\label{q1}
Suppose a collection of $n$ items is to be stored over a set of $m$ servers. Can the items be stored so that any $k$ items are simultaneously accessible by taking at most $t$ items from each server? If so, what amount $N$ of total storage is needed? An optimal solution has the smallest total storage for given parameters $n$, $k$, $m$, and~$t$. 
\end{question}

One can generalize batch codes in different ways.
Zhang et al.~\cite{zhang} consider optimal \emph{multiset batch codes}, which are batch codes in which an item may be requested more than once in a retrieval.
Silberstein~\cite{silberstein2015} generalized the above question by adding a requirement of \defn{redundancy}. 
At each moment, we allow some of the servers to be unavailable (for a practical example, consider that servers may be down for maintenance, or there is an accepted failure rate for the servers). 
If we place a bound, $r$, on the number of servers that may be unavailable at the same time, we are faced with a problem of retrieving each $k$-subset of the items from every collection of $m-r$ servers. 
To ensure that this is possible, intuitively, it will be necessary to store some ``redundant'' copies of items. 
Specifically, Silberstein investigated the existence of \emph{uniform} erasure batch codes by focusing on the following particular question.

%\marginpar{\red{Def from Silberstein?}}
\begin{question}[Silberstein \cite{silberstein2015}]\label{q2a}
Suppose a collection of $n$ items is to be stored over a set of $m$ servers, each of which hosts exactly $\theta$ items. 
At each moment, some number $r$ of the servers may be unavailable.  
Can the items be stored so that any $k$ items are simultaneously accessible from each collection of $m - r$ servers, while taking at most $t$ items from each server?
\end{question}

While Silberstein focuses on uniformity~\cite{silberstein2015}, we direct our attention toward erasure batch codes which are optimal.
To that end, we ask the following question:

\begin{question}\label{q2}
Suppose a collection of $n$ items is to be stored over a set of $m$ servers. At each moment, some number $r$ of the servers may be unavailable.  
Can the items be stored so that any $k$ items are simultaneously accessible from each collection of $m - r$ servers, while taking at most $t$ items from each server?
What amount $N$ of total storage is needed? An optimal solution has the smallest total storage for given parameters $n$, $k$, $m$, $t$, and~$r$.
\end{question}

In Section~\ref{sec:prelim}, we give formal combinatorial definitions and derive several construction lemmas paralleling the results of Bujt\'as and Tuza~\cite{tuza}.  
The two subsequent sections establish the optimal value of $N$ for certain parameter ranges, which are illustrated in
Figure~\ref{fig1}.

In Section~\ref{sec:extremal}, we study the extremal cases when $n$ is large, or small, compared to $m$, extending three results of Paterson, Stinson, and Wei.~\cite{psw}.  
Theorem~\ref{thm:tall} characterizes $N$ when $n \leq m$ and Theorem~\ref{thm:maxk} characterizes $N$ when $k$ is maximal -- the proof and constructions for these are trivial when $r=0$, while the proofs are nontrivial for larger $r$.
Theorem~\ref{thm:nbddbelow} characterizes $N$ when $n \geq (k-1)\binom{m}{r+k-1}$; our proof differs significantly from that given by Paterson, Stinson, and Wei~\cite{psw} when $r=0$.

Section~\ref{sec:gap} studies the ``gap'' between these results, following the approach of Bujt\'as and Tuza~\cite{tuza}. 
We obtain a result characterizing $N$ for certain values of $n$ as small as $\frac{k-1}{r+k-1}\binom{m}{r+k-2}$. 
In these cases, we find that the parameters for which we know optimal solutions are dependent on the existence of maximum packings.
The lower bound is easily expressed when $r=0$ -- in fact the lower bound is precisely $\binom{m}{k-2}$ -- but can be quite complex for larger~$r$.

In Section~\ref{sec:minimum}, we observe that the natural lower bound for the total storage of a erasure batch code, while only achievable in a trivial way when $r=0$, may be obtained by nontrivial means for larger $r$. 
We discuss some such cases, and relate some of them to the existence of graphs with a lower bound on their girth. 

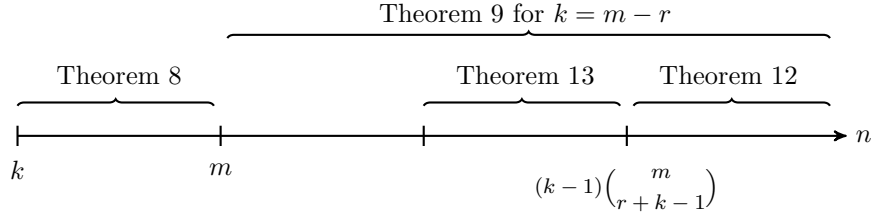
\begin{figure}
\centering
\begin{tikzpicture}[>=stealth',thick,scale=.9]

\draw[->] (0,0)--(12.25,0);
\draw (0,0.15)--(0,-0.15);
\draw (3,0.15)--(3,-0.15);
\draw (6,0.15)--(6,-0.15);
\draw (9,0.15)--(9,-0.15);

\node at (0,-0.5) {$k$};
\node at (3,-0.5) {$m$};
\node at (12.5,0) {$n$};

\node at (5.5,-0.8) {}; % { \footnotesize $\displaystyle\frac{(k-1)}{(k+r-1)}\binom{m}{k+r-2}$};
\node at (9,-0.8) {\footnotesize $\displaystyle (k-1)\binom{m}{r+k-1}$};

\node at (1.5,0.9)  {Theorem~\ref{thm:tall}};
\node at (7.5,0.9)  {Theorem~\ref{thm:fmkr}};
\node at (10.5,0.9) {Theorem~\ref{thm:nbddbelow}};

\node at (7.5, 1.8) {Theorem~\ref{thm:maxk} for $k = m-r$};

\foreach \A/\B in { 2.9/0, 8.9/6, 12/9.1 } 
{ 
\draw [
    thick,
    decoration={
        brace,
        mirror,
        raise=0.4cm
    },
    decorate
] (\A,0)--(\B,0); 
}

\draw [
    thick,
    decoration={
        brace,
        mirror,
        raise=1.3cm
    },
    decorate
] (12,0)--(3.1,0); 

\end{tikzpicture}
\caption{Ranges of the parameter $n$ addressed here, in terms of $k$, $m$, and $r$, always assuming $t = 1$ and that the conditions of Lemma~\ref{lem:exists} are met. Theorem~\ref{thm:tall} applies when $k \leq n \leq m$, and Theorem~\ref{thm:maxk} applies when $n > m$ and $k = m-r$. Theorem~\ref{thm:fmkr} applies to certain values of $n$ less than $(k-1)\binom{m}{r+k-1}$, while Theorem~\ref{thm:nbddbelow} applies when $(k-1)\binom{m}{r+k-1} \leq n$. When $n = (k-1)\binom{m}{r+k-1}$, the constructions in the latter two theorems are the same.}
\label{fig1}
\end{figure}

\section{Erasure combinatorial batch codes}
\label{sec:prelim}

Solutions to Question~\ref{q1} are provided by \defn{combinatorial batch codes}, which were introduced by Paterson, Stinson, and Wei~\cite{psw} and have been the subject of several subsequent papers~\cite{brualdi,tuza3,tuza,tuza2}.  Although we will not directly use combinatorial batch codes here, we state their definition for comparison with the definition of erasure combinatorial batch codes. 

\begin{definition}[Paterson, Stinson, and Wei~\cite{psw}]
 A \defn{combinatorial batch code} with parameters $(n,k,m,t)$, abbreviated CBC or CBC$(n,k,m,t)$, is a multifamily $\mathcal{B}=\{B_1,\ldots,B_m\}$ of $m$ subsets, called \defn{servers}, of a set $[n]=\{1,2,3, \ldots n\}$ of $n$ items, called \defn{files}, such that for each $Y\subseteq X$ with $|Y|\leq k$ there exists subsets $C_i\subseteq B_i$ for which $|C_i|\leq t$ and $Y=C_1\cup C_2\cup \cdots \cup C_m$.
 
The \defn{weight} of a CBC $\mathcal{B}$ is the value
\[
N(\mathcal{B}) = |B_1|+|B_2|+\cdots+|B_m|.
\]
A CBC that obtains the minimal weight (as a function of the other parameters) is \defn{optimal}, and the minimum value is denoted $N(n,k,m,t)$.
\end{definition}

%A CBC with $t=1$ can be related to other combinatorial objects.
%As in~\cite{psw} and~\cite{brualdi}, a $\CBC(n,k,m,1)$ $\mathcal{B} = \{B_1,\dots,B_m\}$ corresponds to an $m\times n$, $\{0,1\}$-valued incidence matrix $A$ such that $a_{ij}=1$ if and only if $B_i$ contains $x_j$.
%Therefore, a $m \times n$ matrix $A$ is the incidence matrix of a $\CBC(n,k,m,1)$ if and only if for any set $C$ of $k$ columns of $A$, the $m\times k$ submatrix $A_C$ determined by the columns in $C$ contains $k$ $1$s with no two $1$s in the same row or column.
%That is, $A$ is the incidence matrix of a $\CBC(n,k.m,1)$ if and only if every $m\times c$ submatrix of $A$ contains a $c$-transversal, for $c \leq k$.
% An equivalent formulation, employed by Tuza~\cite{tuza}, involves the \defn{dual system} of a CBC. 
%
%
To address Question~\ref{q2}, we use a ``redundancy'' parameter, $r$, to measure the number of servers that may be inaccessible at one time in an erasure combinatorial batch code. 
%The codes studied in the following definition
%have also been investigated by Silberstein~\cite{silberstein2015} under the name ``erasure batch codes''. 
% and let $a+[n]=\{a+1,\dots,a+n\}$. 

\begin{definition}[Silberstein~\cite{silberstein2015}]
An \defn{erasure combinatorial batch code with redundancy $r$} with parameters $(n,k,m,t)$ (which we may abbreviate as $r\CBC$, $r\CBC(n,k,m,t)$, or $r\CBC(n,k,m)$ when $t=1$) is a multifamily $\mathcal{B}=\{B_1,\dots,B_m\}$ of $m$ subsets of $[n]$ such that for each $Y\subseteq [n]$ with $|Y|\leq k$ and $J\subseteq[m]$ with $|J|\geq m-r$, there exists subsets $C_j\subseteq B_j$ for each $j\in J$ such that $|C_i|\leq t$ and $Y=\bigcup_{j\in J}C_j$.  

Informally put, this means that for each collection $Y$ of $k$ or fewer files, and each collection $J$ of $m -r$ or more servers, it is possible to obtain all the files in $Y$ from the servers in $J$ while taking no more than~$t$ from each server.

The \defn{weight} of an $r\CBC$ $\mathcal{B}$ is
\[
N(\mathcal{B})  = |B_1|+|B_2|+\cdots+|B_m|,
\]
and an $r\CBC$ that obtains the minimal $N$ (as a function of $n$, $k$, $m$, $t$, and~$r$) is \defn{optimal}.
We denote this minimal value as $N(n,k,m,t;r)$, or simply $N(n,k,m;r)$ if $t=1$.
In this paper, we study the case $t = 1$ exclusively.  \end{definition}

The next lemma establishes the 
basic relations between the remaining parameters that are required for the existence of
an $r\CBC$. For the remainder of the paper, we will always assume that our parameters satisfy the inequalities stated in the lemma.

\begin{lemma}\label{lem:exists}%\marginpar{Parameters}
There is an $r\CBC$ with parameters $(n,k,m,1)$, where $k \geq 1$, if and only if $r< m$ and $k\leq \min\{n, m-r\}$.
\end{lemma}
\begin{proof}
%For the forward direction, it is enough to note that an $m \times n$ matrix of all 1s will represent an $r\CBC$ with the stated parameters, under the stated conditions. 
Store every file on each server.
At any moment there are $m - r$ servers available, each containing every file, and thus we can retrieve any collection of $k$ files for each $k \leq\min\{n,m-r\}$.  

For the reverse direction, we show that the conditions are necessary. If $r \geq m$ then it would be possible for every server to be down, which cannot yield an $r\CBC$ when $k \geq 1$. The inequality $n < k$ is impossible because one cannot retrieve more than the total number of files. The inequality $k > m - r$ is impossible because there will be $m-r $ available servers, and with $t = 1$ we may only take one file from each server.
\end{proof}

%
%As with CBCs, an $m\times n$ incidence matrix $A$ represents an $r$-$\CBC(n,k,m)$ if and only if, for any set $K$ of $k$ columns and any set $L$ of $m-r$ rows of $A$, the $(m-r)\times k$ submatrix $A_{L,K}$ determined by the rows and columns indexed by $L$ and $K$ contains $k$ $1$s with no two $1$s in the same row or column.
%That is, every $(m-r)\times k$ submatrix of $A$ contains a $k$-transversal. We will often identify a matrix with the $r$-CBC that it represents.

%The matrix $A$ is the incidence matrix of an optimal $r$-$\CBC$ if the number of 1s in $A$ is minimal.
%
%Similarly, a dual system $\mathcal{C}$ corresponds to an $r$-$\CBC(n,k,m)$ if and only if for each subset $J=\{n_1,n_2,\dots,n_k\}\subseteq[n]$ and each $r$-subset $R'\subseteq [m]$, 
%there exist $x_{n_i}\in C_{n_i}\backslash R'$ for each $i\in[k]$ such that $x_{n_i}\neq x_{n_j}$ for all $\{i,j\}\subseteq[k]$.
%
%We say that such a dual system belongs to $r$-$\CBC^*(n,k,m)$ or $r$-$\CBC^*(k)$.
%
%Thus for $\mathcal{C}$ to belong to $\CBC^*(k)$, it must satisfy a restricted version of Hall's Marriage Theorem:
%
%
%A majority of our results are stated in terms of the incidence matrix $A$, and we refer to $A$ as a rCBC.
%We use $d$ to denote the number of ``downed servers'' and thus $d=m-l$.

Our first theorem extends Theorem~3 of Bujt\'as and Tuza~\cite{tuza2} to erasure combinatorial batch codes. To achieve this, we use the following extension of Hall's Marriage Theorem, which is also stated by Silberstein~\cite[Theorem 5]{silberstein2015}.  Although the proof method is standard, we have not located a proof in the literature, which leads us to include a proof here.

\begin{lemma}%\marginpar{Should this be a lemma}
\label{thm:hall_extension}
Let $\{A_1,\dots,A_n\}$ be a family of subsets of $M = [n]$ and let $r\geq 0$. Then the following are equivalent:
\begin{enumerate}
\item For each $r$-subset $M'\subseteq M$, there exist distinct elements $a_1,\dots,a_n$ such that $a_i\in A_i\backslash M'$ for each $i\in [n]$.
\item $\left |\bigcup_{j\in J}A_j\right |\geq r+c$ for every $c$-subset $J$ of $[n]$ and $c>0$.
\end{enumerate} 
\end{lemma}

\begin{proof}
Suppose that (i) is satisfied and let $J$ be a $c$-subset of $[n]$.
Let $M'$ be an $r$-subset of $M$.
Then there exist distinct $a_1,a_2,\dots,a_n$ such that $a_i\in A_i\backslash M'$ for each $i\in[n]$.
Hence $|\bigcup_{j\in J}A_j|\geq r+c$ since $\{a_i\mid i\in J\}\cup M'$ is a subset of $\bigcup_{j\in J}A_j$.

Suppose that (ii) is satisfied and let $M'$ be an $r$-subset of $M$.
Observe that for any $c$-subset $J$ of $[n]$ that 
\[\textstyle\left|\bigcup_{j\in J}(A_j\backslash M')\right|=\left|\left(\bigcup_{j\in J}A_j\right)\backslash M'\right|\geq \left|\left(\bigcup_{j\in J}A_j\right)\right|-r\geq (r+c)-r = c.\]
Hence by Hall's Marriage Theorem, there exist distinct elements $a_1,a_2,\dots,a_n$ such that $a_i\in A_i\backslash M'$ for each $i\in[n]$.
\end{proof}

%\begin{proof}
%If condition (i) is satisfied, then $|A_i| > r$ for each $i\in [n]$. 
%Thus, $\left |\bigcup_{j\in J}A_j\right |> r$ for every $c$-subsets $J$ of $[n]$ and $c>0$, and we can have $r$-subsets of $\bigcup_{j\in J}A_j$.
%For each $r$-subset $M' \subseteq \bigcup_{j\in J}A_j \subseteq M$, 
%there exist distinct elements $a_{j_1},\dots,a_{j_c}$ such that $a_{j_i}\in A_{j_i}\backslash M'$ for $j_i\in J$.
%Hence, $\left |\bigcup_{j\in J}A_j\right | \geq r+c$ for every $c$-subset $J$ of $[n]$ and $c>0$.
%
%Assume that condition (i) is unsatisfied.
%That means $|A_i| \leq r$ for some $i \in [n]$, or 
%$| (A_{i_1} \cup A_{i_2}) \backslash M'| \leq 1$ for some $r$-subset $M'\subseteq M$ and $i_{1}, i_{2}\in [n]$. 
%If $|A_i| \leq r$ for some $i \in [n]$, 
%then $\left |\bigcup_{j\in J}A_j\right | < r+1$ for $1$-subset $J=\{i\}$ of $[n]$.
%And, if  $| (A_{i_1} \cup A_{i_2}) \backslash M'| \leq 1$ for some $r$-subset $M'\subseteq M$ and $i_{1}, i_{2}\in [n]$, 
%then $\left |\bigcup_{j\in J}A_j\right | < r+2$ for $2$-subset $J=\{i_1, i_2\}$ of $[n]$.
%This is a contradiction to condition (ii). 
%\end{proof}

In what follows, we present conditions under which a storage arrangement corresponds to an $r\CBC$.
To this end, we represent an $r\CBC(n,k,m)$ $\mathcal{B}$ by a $m \times n$ incidence matrix $A$ 
such that row $i$ of $A$ represents the set $B_i \subseteq [n]$ -- the list of items stored on server $i$. 
We let $A_j$ denote the subset of $[m]$ represented by column $j$ of~$A$ -- that is, $A_j$ is the list of servers on which item $j$ is stored. 
Additionally let $N(A)$ be the number of $1$s that appear in $A$.
Observe that 
\[ N(\mathcal{B}) = |B_1|+|B_2|+\cdots+|B_m| = |A_1| + |A_2| + \cdots + |A_n| = N(A),\]
and hence we say $A$ is optimal if $N(A) = N(n,k,m;r)$. 
 
We now give a characterization of when a storage arrangement is an $r\CBC$ in terms of the rows and columns of its associated incidence matrix.

\begin{theorem}\label{thm:grhc}
Suppose that $A$ is an $m\times n$ matrix with values in $\{0,1\}$. For each $j \leq n$, let $A_j$ be the subset of $[m]$ determined by column $j$ of~$A$. The following are equivalent: 

\begin{enumerate}

\item\label{grhc3}  The matrix $A$ represents an $r\CBC(n,k,m)$. 

\item\label{grhc1} For every $c \in [k]$ and every $c$-subset $J$ of $[n]$, $\left | \bigcup_{j \in J} A_j \right | \geq r+c$.  
%Informally put, for each $c \in [k]$, each collection of $c$ columns of $A$ contains entries from at least $r+c$ rows.
That is, for each $c\in[k]$, the storage of any $c$ items uses at least $r+c$ servers.
%\red{Language without matrices}
\item\label{grhc2}  For every $d$-subset $I$ of $[m]$, with $r \leq d \leq r + k -1$, $|\{ i : A_i \subseteq I \}| \leq d-r$.  
%Informally put, for each collection of $d$ rows of $A$, with $r \leq d \leq r+k-1$, the number of columns whose 1s are completely contained by these rows is at most~$d-r$. 
That is, whenever $r\leq d\leq r+k-1$, any collection of $d$ servers control total access for at most $d-r$ items.
%\red{Language without matrices}
\end{enumerate}
\end{theorem}

\begin{proof}
The implication from (\ref{grhc3}) to (\ref{grhc1}) follows from the definition of an $r\CBC$ with $t = 1$. The implication from (\ref{grhc1}) to (\ref{grhc3}) is a direct application of Lemma~\ref{thm:hall_extension}. 
Therefore, it suffices to prove that  (\ref{grhc1})  and  (\ref{grhc2}) are equivalent.
%
%%%%%%%%%%%%%%%%
%

First, assume that $A$ satisfies condition~(ii). 
Choose $d$ with $r \leq d < r + k $
and let $I$ be a $d$-subset of $[m]$.
Let $J = \{ i : A_i \subseteq I \}$ and let $w = |J|$.  
We want to show that $w \leq d - r$. 
Suppose otherwise: then $w > d - r$, that is, $d < r+ w$.
So we have a collection of more than $d - r$ columns that
are contained in at most $d$ rows, where $r \leq d < r + k$.
So, if we let $c = d - r + 1$, we have a collection of $c$ columns
that are contained in fewer than $c + r$ rows.
Because $r \leq d$, we have $c > 0$. Because $d < r + k$, 
we have $c \leq k$. Thus $c \in [k]$. This contradicts~(ii), which
states that each collection of $c$ columns
must span at least $d = c+r$ rows. Thus we have $w \leq d - r$, 
as desired.

Now assume that $A$ satisfies condition~(iii). 
First, we verify that $A$ satisfies condition~(ii) in the special case
$c = 1$. This follows from the special case of 
(iii) with $d = r$, which says that
for every $r$-subset $I$ of $[m]$ we have
$|\{i : A_i \subseteq I\}| \leq 0$. Thus there is no column 
with fewer than $r+1$ ones, which is precisely the statement of (ii) in the case~$c=1$. 

It remains to prove condition~(ii) for $c \geq 2$. To this end, 
choose $c \in [k]$ with $c \geq 2$ and let $J$ be a $c$-subset of~$[n]$. 
Let $I = \bigcup_{j \in J} A_j$ and let $d = |I|$. 
Note that, by the previous paragraph, each $A_j$ contains at least $r+1$ ones, 
and thus $d \geq r+1$. 

We want to show that $|I| \geq r + c$. Suppose otherwise; then
we have $r + 1 \leq d \leq r + c - 1$, so $r \leq d \leq r + k - 1$, as $c \leq k$. 
Thus, by~(iii), we have 
\[
|\{i : A_i \subseteq I\}| \leq d-r \leq (r+c-1) - r = c-1.
\]
However, we also have
$J \subseteq \{i : A_i \subseteq I\}$, so $|\{i : A_i \subseteq I\}| \geq c$. This is a contradiction, so we conclude $d \geq r + c$, as desired. 
\end{proof}

\begin{lemma}\label{lem:minmax}
Let $A$ be a matrix that represents an $r\CBC(n,k,m)$.
Then the number of $1$s in each column of $A$ is at least $r+1$, and if $A$ is optimal, at most $r+k$.
\end{lemma}

\begin{proof}
By setting $c=1$ in Theorem~\ref{thm:grhc}~(\ref{grhc1}), we see that each column of $A$ has cardinality at least $r+1$.

Now let $A$ represent an $r\CBC(n,k,m)$, and assume without loss of generality that $A_1$ has cardinality greater than $r+k$. We will show that $A$ is not optimal.
Remove an element from $A_1$ and call the resulting matrix $A'$.
%Let $C_j$ and $C'_j$ be subsets of $[m]$ containing the indices of rows incident with column $j$ in $A$ and $A'$, respectively.
Thus $|A'_1|+1=|A_1|> r+k$ and $A'_j=A_j$ for $2\leq j \leq n$.  

To show $A'$ represents an $r\CBC(n,k,m)$, let $J$ be a $c$-subset of~$[n]$ with $c\leq k$. 
If $1\notin J$, then \[\left |\bigcup_{j\in J}A'_j \right | = \left |\bigcup_{j\in J}A_j \right |\geq r+c.\]
If $1\in J$, then 
$$\left |\bigcup_{j\in J}A'_j \right| \geq  \left |A'_1 \right | = |A_1|-1\geq r+k\geq r+c.$$
Thus, by Theorem~\ref{thm:grhc}, $A'$ represents an $r\CBC(n,k,m)$. By construction, $N(A') = N(A) - 1$,  and thus $A$ is not optimal.
\end{proof}
%
%\begin{corollary}\label{cor:bounds}
%Let $A$ be a matrix that represents an $r$-$\CBC(n,k,m)$.
%Then $N(A) \geq (r+1)n$.
%If $A$ is optimal, then $N(A)\leq (r+k)n$.
%\end{corollary}
%
The lemma allows us to bound the weight of an optimal $r\CBC$. 

\begin{corollary}\label{cor:bounds}
The following inequalities hold for $N(n,k,m;r)$:
\[
(r+1)n \leq N(n,k,m;r) \leq (r+k)n.
\]
%Let $A$ be a matrix that represents an $r$-$\CBC(n,k,m)$.
%Then $N(A) \geq (r+1)n$.
%If $A$ is optimal, then $N(A)\leq (r+k)n$.
\end{corollary}

The corollary allows us to prove $N(n,1,m;r)=(r+1)n$. Assuming $k = 1$, form an $r\CBC$ $A$ with a matrix that has  $n$ columns each of cardinality $r+1$. This matrix $A$ is an $r\CBC$ with $N(A) = (r+1)n$, and by the corollary it is impossible to have a smaller value of $N$. 

\section{Extremal Results}\label{sec:extremal}

%\begin{theorem}\label{thm:allresults}\marginbox{Cite  each item in the theorem, remove entries we don't generalize.}
%Let $k, m, n$ be positive integers with $k \leq \min\{n,m\}$. 
%\begin{enumerate}
%\item If $n\leq m$, then $N(n,k,m) = N(n,k,n)=n$.
%Any permutation matrix of order $n$ with an appropriate number of appended zero rows realizes this optimally.
%\item If $m\leq n$, then $N(n,m,m)=mn - m(m-1)$. 
%The identity matrix conjoined with an appropriate all-ones matrix realizes this optimally.

%\item\label{allresults5}
%If $m\leq n$ and $n\geq (k-1)\binom{m}{k-1}$, then $N(n,k,m) = kn - (k-1)\binom{m}{k-1}$.
%Construct a $m\times n$ matrix by taking $k-1$ copies of every possible column vector with $k-1$ ones, then have the remaining columns be any column vector with $k$ ones.
%\end{enumerate}
%\end{theorem}

In this section, we establish the exact value of $N(n,k,m;r)$ for several families of parameter values. Several of our theorems simplify to previously known results when $r = 0$; we note this before each theorem. 

We first consider $r\CBC$s in which there are at least as many servers as files, that is, $n\leq m$.  When $r=0$, the conclusion of the next theorem reduces to $N(n,k,m;0)= n$, which is Theorem~3 of Paterson, Stinson, and Wei~\cite{psw}.  

\newcommand{\Mod}[1]{\ (\text{mod}\ #1)}

\begin{theorem}\label{thm:tall}
Let $n\leq m$.  Then $$N(n,k,m;r) = (r+1)n.$$
\end{theorem}

\begin{proof} 
Let $A$ be the $m\times n$ matrix with columns  corresponding to the sets
$$A_j = \{j+x \Mod m :  x  = 0, 1, \ldots, r \}.$$  
See Figure~\ref{fig:tall} for an example.
We first show that the conditions of Theorem~\ref{thm:grhc}(\ref{grhc1}) are met by establishing for each $c\in[k]$ and each $c$-subset $J$ of $[n]$ that $|\cup_{j\in J}A_j|\geq r+c$ and reaches equality if either $c=m-r$ or $J=\{i,i+1,\dots,i+c-1\}$ for some $i\in[n]$.

This is trivial when $c=1$.
Let $2\leq \ell \leq m-r$ and suppose the above condition holds for any $( \ell-1)$-subset of $[n]$.
Let $x\in J$, $J'=J\backslash\{x\}$, and $U = \cup_{j\in J'}A_j$; we now show that $|U\cup A_x|\geq r+ \ell$.
Note that $|U|\geq r+ \ell-1$, so if $|U|\geq r+ \ell$, then $|U\cup A_x|\geq r+ \ell$.
Suppose instead that $|U|=r+ \ell-1$.
Then $J'=\{i,i+1,\dots,i+ \ell-2\}$ for some $i\in[n]$.
Therefore $U=\{i,i+1,\dots,i+ \ell+r-2\}$.
Observe that since $ \ell+r\leq m$ that $i+ \ell+r-1\notin U$.
If $x\notin U$, then since $x\in A_x$, it follows that $x\in U\cup A_x$ and hence $|U\cup A_x|\geq r+ \ell$.
If $x\in U$, then $x\in\{i+ \ell-1,i+ \ell,\dots,i+ \ell+r-2\}$ and hence $i+ \ell+r-1\in A_x$.
So $|U\cup A_x|\geq r+ \ell$.
Hence by Theorem~\ref{thm:grhc}(\ref{grhc1}) $A$ is an $r\CBC$.
Because $A$ is an $r\CBC$ and $\sum_{j\in [n]} |A_j| = (r+1)n$, $A$ is optimal by Corollary~\ref{cor:bounds}.
\end{proof}

We next consider $r\CBC$s with at least as many files as servers, that is, $n\geq m$, and a maximal number of files being retrieved, $k=m-r$. 
When $r=0$, the conclusion of the next theorem reduces to $N(n,m,m;0)=mn-m(m-1)$, which is Theorem~4 of Paterson, Stinson, and Wei~\cite{psw}.

\begin{figure}
\[A = \begin{bmatrix}1&0&0&1\\1&1&0&0\\1&1&1&0\\1&1&1&1\\0&1&1&1\\0&0&1&1\end{bmatrix}\]
\caption{The matrix $A$ constructed as in the proof of Theorem~\ref{thm:tall} for $m=6, n=4, r=3,$ and $k\leq 3$.}\label{fig:tall}
\end{figure}

\begin{theorem}\label{thm:maxk}
 If $m \leq n$ and $k=m-r$, then $$N(n,m-r,m;r) = mn-m(m-r-1).$$
\end{theorem}

\begin{proof}  Let $A$ be the $m\times n$ matrix with columns given by 
$$A_j = \{j+x \Mod m :  x  = 0, 1, \ldots, r \} \text{ for }j\leq m.$$  
%
%
%$A_j = \{i:i=j+x\pmod m$ for $x\in \{0,\ldots,r\}\}$ for $j\leq m$
%
%
 and $A_j = \{1,2,\ldots,m\}$ for $m<j\leq n$.  See Fig.~\ref{fig:maxk} for an example.  We first show that $A$ is an $r\CBC$.  Let $c\leq m-r$ and let $J\subseteq [n]$ with $|J|=c$.  Then, $|\bigcup_{j\in J} A_j| = m\geq r+c$ if $J\cap\{m+1, \ldots, n\}\neq\emptyset$.  So, suppose $J \cap \{m+1,\ldots,n\} = \emptyset$.  Then, we are only considering the first $m$ columns of $A$, so let $A'$ be the $m\times m$ matrix with columns given by $\{A_1,\ldots,A_m\}$.  We can use the argument from the proof of Theorem~\ref{thm:tall} to show that $A'$ is an $r\CBC$, and thus satisfies Theorem~\ref{thm:grhc}.  Therefore $A$ is an $r\CBC$.  Note that in $A$, for each $d \in [m]$,  $|\{i:d \in A_i\}|=n-m+r+1$.

Because $A$ is an $r\CBC$ with $N(A)=m(n-m+r+1)$, we know that $N(n,m-r,m;r)\leq   m(n-m+r+1)$.    Suppose that $N(n,m-r,m;r)<m(n-m+r+1)$.  Let $B$ be an $r$\CBC\ with $N(B)=N(n,m-r,m;r)$.  Then there must be some $d \in [m]$ for which $|\{i:d \in B_i\}|<n-m+r+1$.   Let $J \subseteq [n]\setminus \{i:d\in B_i\}$ with $|J|=m-r$.  Such a  $J$ exists because 
$$\left |[n]\setminus \{i:d\in B_i\} \right |>n-(n-m+r+1) = m-r-1.$$
 Then $|\bigcup_{j\in J}B_j|<(m-r)+r$, because $d\notin B_j$ for any $j \in J$.  Therefore, by Theorem~\ref{thm:grhc}, $B$ is not an $r\CBC$. 

We can thus conclude that $A$ is optimal and \[N(n,m-r,m;r) = m(n-m+r+1).\qedhere\]
\end{proof}

\begin{figure}
\[A=\begin{bmatrix} 1&0&0&1&1&1&1\\1&1&0&0&1&1&1\\1&1&1&0&0&1&1\\1&1&1&1&0&0&1\\0&1&1&1&1&0&1\\0&0&1&1&1&1&1\end{bmatrix}\]
\caption{The matrix $A$ constructed as in the proof of Theorem~\ref{thm:maxk} when $m=6, n=7, r=3,$ and $k=3$.}\label{fig:maxk}
\end{figure}
We now examine what occurs when $n$ is large.
We first prove two technical lemmas, which parallel lemmas of Bujt\'as and Tuza~\cite{tuza3}.%\marginpar{Specify which lemmas?}

\begin{lemma}
\label{lem:move_ones}
Let $A$ be an $r\CBC$, and assume there are $i, j \in[n]$ with $A_i \subseteq A_j$. 
Let $R$ be a nonempty set with $R \subseteq A_j \setminus A_i$.  Replacing columns $A_i$ and $A_j$ with $A'_i = A_i \cup R$ and $A'_j = A_j \setminus R$, respectively, produces a matrix $A'$ which is an $r\CBC$ with the same weight as $A$.
% Let $A$ be an $r\CBC$, and assume there are $i, j \in[n]$ with $A_i \subseteq A_j$ and $|A_i| \leq |A_j|$. Let $R$ be a nonempty set with $R \subseteq A_j \setminus A_i$.  Replacing columns $A_i$ and $A_j$ with $A'_i = A_i \cup R$ and $A'_j = A_j \setminus R$, respectively, produces a matrix $A'$ which is an $r\CBC$ with the same weight as $A$.
\end{lemma}

\begin{proof}
Assume that $A$, $A'$, $i$, and $j$ are as in the statement of the theorem.
We wish to prove that $A'$ satisfies condition~(iii) of Theorem~\ref{thm:grhc}. 
To this end, assume that $I$ is a $d$-subset of $[m]$. 
%\red{\st{It is sufficient to show that the number $\sigma'$ of values of $k$ for which $A'_k \subseteq I$ is no more than the number $\sigma$ of values of $k$ for which $A_k \subseteq I$.}
It is sufficient to show that $\sigma' \leq\sigma$ where $\sigma$ and $\sigma'$ denote the number of columns of $A$ and $A'$, respectively, which are contained in $I$.

The proof has several cases. 
If $A_j \subseteq I$, then $A_i \subseteq I$ as well, because $A_i \subseteq A_j$. 
Thus $A'_i$ and $A'_j$ are also subsets of $I$. 
Thus if $A_j \subseteq I$ then $\sigma  = \sigma'$.

If $A_i \not \subseteq I$ then, because $A_i \subseteq A_j$, we have that $A_j \not \subseteq I$. 
Thus neither of $A'_i$ and $A'_j$ is a subset of $I$, so again $\sigma = \sigma'$.

Because $A_i \subseteq A_j$, there is only one additional case, which occurs when $A_i \subseteq I$ but $A_j \not \subseteq I$. 
There are two subcases: if $A_j \setminus I \subseteq R$, then $A'_j \subseteq I$ and $A'_i \not \subseteq I$. 
Otherwise, if $A_j \setminus I \not \subseteq R$, then neither $A'_i$ nor $A'_j$ is a subset of~$I$. 
Thus, in both subcases, we have $\sigma' \leq \sigma$. 
So $A'$ is an $r\CBC$ and since $|A_i|+|A_j|=|A_i'|+|A_j'|$, it follows that $A'$ has the same weight as $A$.
\end{proof}

\begin{lemma}
\label{lem:types}
For every matrix $A$ representing an optimal $r\CBC(n,k,m)$, there exists a matrix $A'$ also representing an optimal $r\CBC(n,k,m)$ for which %\red{\st{one of the two properties hold:}
either each column of $A'$ has weight at most $r+k-1$, or each column of $A'$ has weight $r+k-1$ or $r+k$.
%\begin{enumerate}
%\item\label{lem:prop1}\red{\st{$r+1\leq |A'_j| \leq r+k-1$ for each $j \leq n$, or}} 
%\item\label{lem:prop2}\red{\st{$r+k-1\leq |A'_j| \leq r+k$ for each $j\leq n$.}}
%\end{enumerate}

\begin{proof}
By Lemma~\ref{lem:minmax}, if $A$ represents an optimal $r\CBC$ then all columns of $A$ have cardinality  at most $r+k$.
%Assume $A$ does not satisfy properties (\ref{lem:prop1}) or (\ref{lem:prop2}).
Assume $A$ has a column $C'$ with cardinality $r+k$ and a column $C$ with cardinality $j$, where $r+1\leq j \leq r+k-2$.
%not of type $[k+r-1,k+r]$ or $[r+1,k+r-1]$.
%Then $A$ contains columns $C$ and $C'$ of cardinality $j$ $(r+1\leq j\leq r+k-2)$ and $r+k$, respectively.
%%New
Observe that if an $r\CBC$ has a column of cardinality $r+k$, one can replace the column with any other column of cardinality $r+k$, and still satisfy the $r\CBC$ property.
So we can assume that 
%\red{\st{$C\subseteq C'$} 
$C'$ contains all $1$s of $C$.
%%EndNew
%Replacing $C'$ with any set of cardinality $r+k$ produces an $r\CBC$, so replace $C'$ with any column $C''$ for which $C\subseteq C''$.
By our previous result, we can then produce another $r\CBC$ with one fewer column of cardinality $r+k$, but with the same weight.

Proceed inductively until either the resulting $r\CBC$ has no columns of cardinality $r+k$ or there are no columns of cardinality less than $r+k-1$.
In these cases, we produce either an $r\CBC$ with all columns having cardinality at most $r+k-1$, or an $r\CBC$ with all columns having cardinality $r+k-1$ or $r+k$ , respectively.
At every step, the weight of the $r\CBC$ remains unchanged, and therefore the final $r\CBC$ is optimal.
\end{proof}
\end{lemma}

When $r = 0$, the conclusion of our next result simplifies to $N(n,k,m;0) = kn-(k-1)\binom{m}{k-1}$, which is Theorem~8 of Paterson, Stinson, and Wei~\cite{psw}.

\begin{theorem}\label{thm:nbddbelow}Let $r\geq 0$ and $k\leq m-r$ be integers.  If $n \geq (k-1)\binom{m}{r+k-1}$, then 
\[N(n,k,m;r) = (r+k)n - (k-1)\binom{m}{r+k-1}.\]
\end{theorem}

\begin{proof}
Following Corollary~\ref{cor:bounds} of Section~\ref{sec:prelim}, we showed that $N(n,1,m;r)=(r+1)n$, so the result holds when $k=1$.
We may thus assume that $k\geq 2$.
Set $M = \binom{m}{r+k-1}$.  
%Let $A$ be the $m\times n$ matrix constructed as follows.  
Let $\{A_i\mid 1\leq i\leq (k-1)M\}$ consist of $k-1$ copies of each possible subset of $[m]$ with cardinality $r+k-1$, and let $\{A_i\mid (k-1)M+1 \leq i \leq n\}$ be a set of any $n-(k-1)M$ subsets of $[m]$ with cardinality $r+k$.
%Let $A_1, \ldots, A_M$ be all the subsets of $[m]$ of size $r+k-1$.  Then let $A_{M+j}$ be a copy of $A_{j\Mod{M}}$ for each $1\leq j\leq (k-2)M$.  For each $i$, $(k-1)M+1\leq i\leq n$, let $A_i$ be some $(r+k)$-subset of $[m]$.  
Let $A$ be the $m\times n$ matrix defined by the union of these two sets.  For an example, see Figure~\ref{fig:nbddbelow}.  It follows that 
\begin{align*}
N(A) &= (r+k-1)(k-1)M+(r+k)[n-(k-1)M] \\
&= (r+k)n - (k-1)M.
\end{align*}

Let $c \leq k$ and $J \subseteq[n]$ with $|J|=c$.  If there is some $j' \in J$ with $j'\geq (k-1)M+1$, then $\left|\bigcup_{j\in J}A_j\right|\geq |A_{j'}|=r+k \geq r+c$.  Suppose then that $J \subseteq \{1,\ldots,(k-1)M\}$.  Then, for each $j$, $|A_j| = r+k-1$, so if $c<k$, $\left|\bigcup_{j\in J}A_j\right|\geq r+k-1\geq r+c$.  Suppose then that $c=k$.  Consider $j_1, j_2 \in J$ with $A_{j_1} \neq A_{j_2}$.  Such a pair of sets must exist because there are only $k-1$ copies of each subset of $[m]$ with cardinality $r+k-1$. Thus $\left|\bigcup_{j\in J}A_j \right| \geq |A_{j_1}\cup A_{j_2}| \geq r+k-1+1  = r+c$.  Thus, in all cases, Theorem~\ref{thm:grhc}~(ii) is satisfied and $A$ is an $r\CBC$.

Let $B$ be an optimal $r\CBC(n,k,m)$.
Without loss of generality, we can assume $B$ is of type (i) or (ii) as outlined in Lemma \ref{lem:types}.
Suppose that $B$ is of type (i), so that $r+1\leq |B_j|\leq r+k-1$ for each $j \leq n$.
Because $B$ is an $r\CBC$, every $(r+k-1)$-subset of $[m]$ contains at most $k-1$ columns of $B$, meaning that $n\leq (k-1)M$ and thus $n=(k-1)M$.
%%New
Let $\mathcal{C}$ be the set of ordered pairs
\[ \mathcal{C} = \{ (B_j,I)\mid k\in[n],\ B_k\subseteq I \subseteq[m],\ |I|=r+k-1\}.\]
Observe that for each $j\in[n]$, since $|B_j|\leq r+k-1$, there is at least one such ordered pair in $\mathcal{C}$ including $B_j$.
Therefore $|\mathcal{C}|\geq n$.
However, for each $I\subseteq[m]$ with $|I|=r+k-1$, there are at most $k-1$ columns of $B$ which $I$ contains.
So $|\mathcal{C}| \leq (k-1)M$.
Therefore $|\mathcal{C}|=(k-1)M$ and hence each column of $B$ is contained in exactly one subset of cardinality $r+k-1$.
So each column of $B$ has cardinality $r+k-1$, and it follows that $N(B) = (r+k)n-(k-1)M$.
%%
%Observe that for each $
%Consider the set $\mathcal{C}$ of ordered pairs $(B_j,R)$, where $j\in[n]$ and $R$ is a $(r+k-1)$-subset of $[m]$.
%Assume that there exists a column of $B$ with cardinality less than $r+k-1$.
%Then $|\mathcal{C}|>(k-1)M$.
%So there exists a $(r+k-1)$-subset $R\subseteq[m]$ so that $\{(B_j,R)\mid j\in[n],B_j\subseteq R\}$ has at least $k$ columns, which is a contradiction.
%Therefore each column of $B$ has cardinality $r+k-1$ and hence $N(B)=(r+k)n-(k-1)M$.

Now, suppose that $B$ is of type (ii), that is, $r+k-1\leq |B_j|\leq r+k$ for each $j\leq n$.
Because $B$ is an $r\CBC$, the maximal number of columns of $B$ with cardinality $r+k-1$ is $(k-1)M$.
Therefore 
\begin{align*}
N(B)&\geq (r+k-1)(k-1)M+(r+k)[n-(k-1)M] \\
&= (r+k)n - (k-1)M. \qedhere
\end{align*}
%Then $B$ has a collection of $(k-1)\binom{m}{r+k-1}+1$ columns, each with size exactly $r+k-1$.  Thus, there exists some $(r+k-1)$-subset $J$ of $[m]$ with $|\{i: B_i\subseteq J\}|\geq k$, so, by Theorem~\ref{thm:grhc}, $B$ is not an $r$-CBC.  Therefore, $N(n,k,m;r) = (r+k)n - (k-1)\binom{m}{r+k-1}$.
\end{proof}

\begin{figure}
\[A=\begin{bmatrix} 1&1&1&0&0&0&1&0\\1&0&0&1&1&0&1&1\\0&1&0&1&0&1&1&1\\0&0&1&0&1&1&0&1\end{bmatrix}\]
\caption{The matrix $A$ constructed as in the proof of Theorem \ref{thm:nbddbelow} when $m=4, k=2, r=1,$ and $n=8\geq (k-1)\binom{m}{r+k-1}$.}\label{fig:nbddbelow}
\end{figure}

\section{Narrowing the gap}\label{sec:gap}

In the previous section, for fixed $m$, $k$, and $r$, we established $N(n,k,m;r)$ when $n\leq m$ and $n\geq (k-1)\binom{m}{r+k-1}$.
In this section, we address the ``gap'' between these results, and establish $N(n,k,m;r)$ for values of $n$ immediately below the latter boundary, as illustrated in Figure~\ref{fig1}.

We begin by identifying $N(n,2,m;r)$ for all possible parameters $n$, $m$, and $r$.
If $n\geq \binom{m}{r+1}$, then by Theorem~\ref{thm:nbddbelow}, $N(n,2,m;r)=(r+2)n-\binom{m}{r+1}$.
Suppose that $n\leq \binom{m}{r+1}$.
Let $A$ be a $m\times n$ matrix whose columns are distinct $(r+1)$-subsets of $[m]$.
Then $A$ is an $r\CBC(n,2,m)$ and $N(A)=(r+1)n$.
So, by Corollary~\ref{cor:bounds}, $N(n,2,m;r)=(r+1)n$.
Having completed the case $k = 2$ we assume that $k\geq 3$ for the remainder of this section. 
The goal of this section is to prove the following theorem:

\begin{theorem}
Let $r\geq 0$, $k\geq 3$, and $m\geq r+k$.  
If 
\begin{equation}
\label{eqn:midinterval}
(k-1)\tbinom{m}{r+k-1}-(m-r-k+1)\cdot F(k,m,r) \leq n \leq (k-1)\tbinom{m}{r+k-1}
\end{equation}
for an appropriate constant $F(k,m,r)\leq \frac{k-1}{r+k-1}\binom{m}{r+k-2}$, then
\[ N(n,k,m;r) = (r+k-1)n - \left\lfloor\dfrac{(k-1)\binom{m}{r+k-1}-n}{m-r-k+1}\right\rfloor.\]
\label{thm:fmkr}
\end{theorem}

Before proving the theorem, we first prove some technical lemmas that are in line with the arguments of Paterson, Stinson, and Wei~\cite{psw} as well as Bujt\'as and Tuza~\cite{tuza2}.
We then define $F(k,m,r)$, and follow with the proof of Theorem \ref{thm:fmkr}.
We conclude the section with some concepts from design theory, then show that the existence of a design with certain parameters significantly reduces the complexity of the lower bound in Theorem~\ref{thm:fmkr}.

The proof method of the following lemma is similar to Theorem~8 of Paterson, Stinson, and Wei~\cite{psw}, which gives an upper bound on the number of columns within a range of cardinalities.
\begin{lemma}
\label{prop:upperbound}
Let $A$ represent an $r$\CBC\ and 
for each $i > r$,
let $\ell_i$ denote the number of columns of $A$ with cardinality $i$.
%, for $r+1\leq i\leq r+k$.
Then 
\[ \sum_{i=r+1}^{r+k-1}\ell_i\binom{m-i}{r+k-1-i}\leq (k-1)\binom{m}{r+k-1}.\]
\end{lemma}
\begin{proof}
We count in two ways the number $a$ of pairs $(R,A_j)$ where $R\subseteq[m]$, $|R|=r+k-1$, $j\in[n]$, and $A_j\subseteq R$.
By Theorem~\ref{thm:grhc}, every $(r+k-1)$-subset $R$ of $[m]$ contains at most $k-1$ columns of $A$.
Therefore $a\leq (k-1)\binom{m}{r+k-1}$.
Furthermore, for each column of $A$ with cardinality $i$, there are $\binom{m-i}{r+k-1-i}$ $(r+k-1)$-subsets $R$ of $[m]$ containing it.
Therefore 
\[ a = \sum_{\substack{R\subseteq[m]\\|R|=r+k-1}}
%R\subseteq\binom{[m]}{k+r-1}}
|\{j:A_j\subseteq R\}| = \sum_{i=r+1}^{r+k-1}\ell_i\binom{m-i}{r+k-1-i}.\]
The result follows.
\end{proof}

The next result identifies the maximum number of columns of cardinality $r+k-1$ that can be appended to an $r$\CBC\ to obtain a larger $r$\CBC.

\begin{lemma}%\marginpar{Replaced $t$ with $b$}
\label{prop:extension}
Let $A$ represent an $r\CBC(n,k,m)$.
Then $A$ can be extended to an $r\CBC(n+b,k,m)$ $A'$ with $b$ additional columns of cardinality $r+k-1$ if and only if 
\begin{equation}
b \leq (k-1)\binom{m}{r+k-1} - \sum_{i=r+1}^{r+k-1}\ell_i\binom{m-i}{r+k-1-i},
\label{eqn:prop3}
\end{equation}
where for each $i>r$, $\ell_i$ denotes the number of columns in $A$ of cardinality $i$.%, for $r+1\leq i\leq r+k-1$.
\end{lemma}

\begin{proof}
It is sufficient to show the result holds when $A$ has no column with cardinality greater than $r+k-1$.

Suppose such an extension is possible and let $\ell_i'$, for $r+1\leq i\leq r+k-1$, denote the number of columns in $A'$ of cardinality $i$.
Observe that $\ell_{r+k-1}'=\ell_{r+k-1}+b$ and $\ell_i'=\ell_i$ for all other~$i$.
By Lemma \ref{prop:upperbound}, we have that 
\[ \sum_{i=r+1}^{r+k-1}\ell_i'\binom{m-i}{r+k-1-i} \leq (k-1)\binom{m}{r+k-1} \]
and therefore 
\[ 
%\sum_{i=r+1}^{r+k-1}\ell_i'\binom{m-i}{r+k-1-i} =
 b + \sum_{i=r+1}^{r+k-1}\ell_i\binom{m-i}{r+k-1-i}\leq (k-1)\binom{m}{r+k-1}. \]
%So (\ref{eqn:prop3}) follows.

Now suppose that (\ref{eqn:prop3}) holds.
Let $\mathcal{C}$ be the set of all possible columns of cardinality $r+k-1$.
Let $C\in\mathcal{C}$.  
Since there are at most $k-1$ columns of $A$ contained in $C\in \mathcal{C}$, we can define $b_C \geq0$ so that there are $k-1-b_C$ columns of $A$ contained in $C$.
%Define $t_C$ so that there are $k-1-t_C$ columns of $A$ which are contained in $C$.
%Since there are at most $k-1$ columns of $A$ which are contained in $C$, $t_C\geq 0$.
Hence we can append up to $b_C$ copies of $C$ to $A$, and the resulting matrix will be an $r\CBC$.
Then the lemma follows if we show that $b\leq \sum_{C\in\mathcal{C}}b_C$.

Recall that each column of $A$ with cardinality $i$ is contained in $\binom{m-i}{r+k-i-1}$ columns of $\mathcal{C}$.
Therefore by an argument similar to the one above,
\[\sum_{i=r+1}^{r+k-1}\ell_i\binom{m-i}{r+k-1-i} 
= \sum_{C\in\mathcal{C}}(k-1-b_C) 
= (k-1)\binom{m}{r+k-1}-\sum_{C\in\mathcal{C}}b_C.\]
The result follows.
%\[\sum_{C\in\mathcal{C}}(k-1-t_C) = (k-1)\binom{m}{k+r-1} - \sum_{C\in\mathcal{C}}t_C\]
%\[\sum_{C\in\mathcal{C}}t_C = (k-1)\binom{m}{k+r-1} - \sum_{C\in\mathcal{C}}(k-1-t_C)  \]
%\[\sum_{C\in\mathcal{C}}t_C = (k-1)\binom{m}{k+r-1} - \sum_{i=r+1}^{k+r-1}l_i\binom{m-i}{k+r-1-i}  \]
%\[\sum_{C\in\mathcal{C}}t_C \geq t\]
%\[\begin{array}{rc>{\displaystyle}l}
%t &\leq& (k-1)\binom{m}{k+r-1} - \sum_{i=r+1}^{k+r-1}l_i\binom{m-i}{k+r-1-i} \\
%&=& (k-1)\binom{m}{k+r-1} - \sum_{C\in\mathcal{C}}(k-1-t_C) 
% = \sum_{C\in\mathcal{C}}t_C\\
%\end{array}\]
\end{proof} 

We will require the number-theoretic result given in Lemma~1 of Bujt\'as and Tuza~\cite{tuza3}, which we restate here.

\begin{lemma}[Bujt\'as and Tuza~\cite{tuza3}]
\label{lem:tuza}
For any three integers $i$, $p$, and $m$ satisfying $1\leq i\leq p\leq m-1$, the following inequality holds:
\[\left\lfloor \dfrac{\binom{m-i}{p-i}-1}{m-p}\right\rfloor \geq p-i.\]
\end{lemma}

We now define $F(k,m,r)$, prove existence of an upper bound on $F(k,m,r)$, and provide and example for computing $F(k,m,r)$.

\begin{definition}\label{def:f}
For parameters $k\geq 3$, $r\geq 0$, and $m\geq r+k$, let $F(k,m,r)$ be the largest $n$ such that an $r\CBC(n,k,m)$ exists in which each column has cardinality $r+k-2$.
Such an $r\CBC(n,k,m)$ and $F(k,m,r)$ are closely related to \defn{packing designs} and \defn{packing numbers}~\cite{millsmullin}, which will be discussed at the end of the section.
\end{definition}

\begin{example}
\label{ex:fmkr1}
Let $m=5$, $r=1$, and $k=3$.
A $1\CBC(6,3,5)$ exists, as shown in in Figure \ref{fig:fmkr}.
Observe that each column contains $r+k-2=2$ 1s.
Hence $F(3,5,1)\geq 6$.
\end{example}

\begin{figure}
\[\begin{bmatrix}
0&0&0&0&1&1 \\
0&0&1&1&0&1 \\
0&1&0&1&0&0 \\
1&0&1&0&0&0 \\
1&1&0&0&1&0 \\
%1&1&1&1&0&0 \\
%1&1&0&0&1&0 \\
%1&0&1&0&1&1 \\
%0&1&0&1&1&1 \\
%0&0&1&1&0&1 \\
\end{bmatrix}\]
\caption{A $1\CBC(6,3,5)$ illustrating that $F(3,5,1)\geq 6$. See Example \ref{ex:fmkr1}.}
\label{fig:fmkr}
\end{figure}

\begin{lemma}
Letting $F(k,m,r)$ be as in Definition~\ref{def:f}, we have 
\[F(k,m,r) \leq \dfrac{k-1}{r+k-1}\dbinom{m}{r+k-2}.\]
\label{lem:fmkr_inequality}
\end{lemma}
\begin{proof}
Let $P$ be an $r\CBC(n,k,m)$ with $n=F(k,m,r)$ and suppose the columns of $P$ have cardinality $r+k-2$.
We enumerate the set 
\[ \mathcal{C}=\{(C,I)\mid C\mbox{ is a column of $P$}, C\subseteq I\subseteq[m], |I|=r+k-1\} \]
in two ways.
For any given column $C$ of $P$, we have that $|C|=r+k-2$ and so there are $m-(r+k-2)$ possible subsets $I\subseteq[m]$ for which $|I|=r+k-1$ and $C\subset I$.
Therefore $|\mathcal{C}| = F(k,m,r)\cdot(m-r-k+2)$.

For any given $I\subseteq[m]$ with $|I|=r+k-1$, since $P$ is an $r\CBC(n,k,m)$, there are at most $k-1$ columns $C$ of $P$ for which $C\subset I$.
Therefore $|\mathcal{C}| \leq \binom{m}{r+k-1}\cdot (k-1)$.
%
%Let $\mathcal{C}$ be the column set of an $r\CBC(n,k,m)$ for which $n=F(k,m,r)$.
%Consider the set $\{(C,I)\mid C\in\mathcal{C}, C\subseteq I\subseteq[m], |I|=r+k-1\}$.
%For each $C\in\mathcal{C}$, there are $m-(r+k-2)$ subsets $I$ of cardinality $r+k-1$ containing $C$.
%For each $I\subseteq[m]$ of cardinality $r+k-1$, because $\mathcal{C}$ consists of the columns of an $r\CBC$, there are at most $k-1$ columns of $\mathcal{C}$ contained in~$I$.
So 
$F(k,m,r)\cdot (m-(r+k-2)) \leq \binom{m}{r+k-1}\cdot (k-1).$
The result follows.
\end{proof}

\begin{example}
\label{ex:fmkr2}
Let $m=5$, $r=1$, and $k=3$.
From Lemma \ref{lem:fmkr_inequality}, observe that $F(3,5,1)\leq 20/3 < 7$.
Therefore, since we showed that $F(3,5,1)\geq 6$ in Example \ref{ex:fmkr1}, it follows that $F(3,5,1)=6$.
\end{example}

We now prove Theorem \ref{thm:fmkr}, letting $F(k,m,r)$ take the value from
Definition~\ref{def:f}.
%\begin{theorem}\label{thm:gap}
%If $r\geq 0$, $m\geq k+r\geq 3$ (?), and 
%\[(k-1)\tbinom{m}{k+r-1}-F(m,k,r)(m-k-r+1)\leq n \leq (k-1)\tbinom{m}{k+r-1},\] 
%\[0 \leq (k-1)\tbinom{m}{k+r-1}-n \leq (m-k-r+1)F(m,k,r),\] 
%then
%\[ N = (k+r-1)n - \left\lfloor\dfrac{(k-1)\binom{m}{k+r-1}-n}{m-k-r+1}\right\rfloor.\]
%\label{thm:fmkr}
%\end{theorem}\marginbox{Question mark in the theorem statement}
%
\begin{proof}[Proof of Theorem \ref{thm:fmkr}]
Suppose that $m=r+k$.
Then $N(n,k,m;r)=mn-m(m-r-1)$ by Theorem~\ref{thm:maxk}, and our formula gives 
\[\begin{array}{rcl}
N &=& (m-1)n - \left\lfloor(k-1)m-n\right\rfloor\\
&=& mn-km+m \\
&=& mn-(m-r)m+m \\
&=& mn-(m-r-1)m. \\
\end{array}\]
So the formula holds.
Now assume that $m>r+k$.
Let $P$ be an $r\CBC(n,k,m)$ with $n=F(k,m,r)$ and suppose the columns of $P$ have cardinality $r+k-2$.
Let $\mathcal{C}$ be the set of all columns contained in $[m]$ with cardinality $r+k-1$.
Let $$x=\left\lfloor\dfrac{(k-1)\binom{m}{r+k-1}-n}{m-r-k+1}\right\rfloor.$$
It follows from the restriction on $n$ that $0\leq x \leq F(k,m,r)$.
Let $B$ be an $r\CBC(x,k,m)$ consisting of $x$ columns from $P$.
By Lemma~\ref{prop:extension}, $B$ can be extended by appending up to $(k-1)\binom{m}{r+k-1} - x(m-r-k+2)$ columns 
%with cardinality $r+k-1$
from $\mathcal{C}$, and 

%\marginbox{Once we remove any un-needed lines, check the width and page break of the long display}
\begin{align*}
&(k-1)\binom{m}{r+k-1} - x(m-r-k+2){\rule[-15pt]{0pt}{0pt}}\\
&= (k-1)\binom{m}{r+k-1} - \left\lfloor\dfrac{(k-1)\binom{m}{r+k-1}-n}{m-r-k+1}\right\rfloor(m-r-k+2){\rule[-25pt]{0pt}{0pt}}\displaybreak[0]\\
%&\geq (k-1)\binom{m}{r+k-1} - \left\lfloor\dfrac{(k-1)\binom{m}{r+k-1}-n}{m-r-k+1}(m-r-k+2)\right\rfloor{\rule[-25pt]{0pt}{0pt}}\displaybreak[0]\\
&\geq (k-1)\binom{m}{r+k-1} - \left\lfloor\dfrac{(k-1)\binom{m}{r+k-1}-n}{m-r-k+1}+(k-1)\binom{m}{r+k-1}-n\right\rfloor{\rule[-25pt]{0pt}{0pt}}\displaybreak[0]\\
&= (k-1)\binom{m}{r+k-1} - \left\lfloor\dfrac{(k-1)\binom{m}{r+k-1}-n}{m-r-k+1}\right\rfloor-(k-1)\binom{m}{r+k-1}+n{\rule[-15pt]{0pt}{0pt}}\displaybreak[0]\\
&= n-x.
\end{align*}
Let $B'$ be an extension of $B$ obtained by appending the appropriate $n-x$ columns from $\mathcal{C}$.
Then $N(B')=(r+k-1)n-x$ and hence $N(n,k,m;r)\leq (r+k-1)n-x$.

Let $A$ be an optimal $r\CBC(n,k,m)$.
In what follows, we show that $N(A)\geq (r+k-1)n - x$.
By Lemma~\ref{lem:types}, we may assume $A$ has that all of its columns are of cardinality at most $r+k-1$ or all of its columns have cardinality $r+k-1$ or $r+k$.
If $A$ is of the latter type, then $|A_j|\geq r+k-1$ for each $j\in[n]$, and therefore $N(A)\geq(r+k-1)n$.

Suppose now that each column of $A$ has cardinality at most $r+k-1$; that is $r+1\leq |A_j|\leq r+k-1$ for each $j\in[n]$.
It is sufficient to show that
\[ N(A) = \sum_{i=r+1}^{r+k-1} i\ell_i = (r+k-1)n -\sum_{i=r+1}^{r+k-1}(r+k-1-i)\ell_i \geq (r+k-1)n-x,\]
where again $\ell_i$ $(r+1\leq i\leq r+k-1)$ denotes the number of columns of $A$ with cardinality $i$.
Hence it is sufficient to show that
\begin{equation}
\label{eqn:lowerbound}
\sum_{i=r+1}^{r+k-1}(r+k-1-i)\ell_i \leq x.
\end{equation}
By Lemma~\ref{prop:upperbound}, we have that 
\[\sum_{i=r+1}^{r+k-1}\ell_i\binom{m-i}{r+k-1-i} \leq (k-1)\binom{m}{r+k-1}.\]
%By splitting off the term with $i=r+k-1$, this inequality becomes
%\[ \ell_{r+k-1} + \sum_{i=r+1}^{r+k-2}\ell_i\binom{m-i}{r+k-1-i} \leq (k-1)\binom{m}{r+k-1}.\]
Since $\ell_{r+k-1} = n - (\ell_{r+1} + \cdots + \ell_{r+k-2})$, we can substitute:
\[ n-(\ell_{r+1} + \cdots + \ell_{r+k-2}) + \sum_{i=r+1}^{r+k-2}\ell_i\binom{m-i}{r+k-1-i} \leq (k-1)\binom{m}{r+k-1}.\]
We can move the $n$ over to the right side of the inequality, incorporate the $\ell_i$s in the summation, and divide both sides by $m-r-k+1$ (which is at least $1$):
%\[ -(l_{r+1} + \cdots + l_{r+k-1}) + \sum_{i=r+1}^{r+k-2}l_i\binom{m-i}{r+k-1-i} \leq \binom{m}{r+k-1}-n.\]
%\[ \sum_{i=r+1}^{r+k-2}l_i\left(\binom{m-i}{r+k-1-i}-1\right) \leq \binom{m}{r+k-1}-n.\]
\[ \sum_{i=r+1}^{r+k-2}\ell_i\left(\dfrac{\dbinom{m-i}{r+k-1-i}-1}{m-r-k+1}\right) \leq \dfrac{(k-1)\dbinom{m}{r+k-1}-n}{m-r-k+1}.\]
Taking the floor on the left for each term will produce something smaller than the floor of the term on the right, and the result follows.
%\[ \sum_{i=r+1}^{r+k-2}\ell_i\left\lfloor\dfrac{\dbinom{m-i}{r+k-1-i}-1}{m-r-k+1}\right\rfloor \leq \left\lfloor\dfrac{(k-1)\dbinom{m}{r+k-1}-n}{m-r-k+1}\right\rfloor.\]
%By Lemma~\ref{lem:tuza} and the definition of $x$:
%\[ \sum_{i=r+1}^{r+k-2}l_i(r+k-1-i) \leq \left\lfloor\dfrac{\dbinom{m}{r+k-1}-n}{m-k-r+1}\right\rfloor.\]
%\[ \sum_{i=r+1}^{r+k-2}\ell_i(r+k-1-i) \leq x.\]
%So (\ref{eqn:lowerbound}) follows, proving our result.
\end{proof}

Observe that the lower bound for the values of $n$ in Theorem \ref{thm:fmkr} is dependent $F(k,m,r)$.
In what follows, we associate $F(k,m,r)$ to maximal packings with appropriate parameters and, in the case when a design exists with certain properties, we significantly simplify the bounds on $n$ in Theorem~\ref{thm:fmkr}.
We now classify when equality is reached in Lemma \ref{lem:fmkr_inequality}, which in turn minimizes and simplifies the lower bound of the interval (\ref{eqn:midinterval}) in Theorem \ref{thm:fmkr}.

\begin{definition}
Let $X$ be a set and let $\mathcal{B}$ be a family of subsets of $X$.
Recall that the ordered pair $(X,\mathcal{B})$ is a $t\dash(v,k,\lambda)$ \defn{design} if $|X|=v$, $|B|=k$ for each $B\in\mathcal{B}$, and  each $t$-subset of $X$ is contained in exactly $\lambda$ sets in $\mathcal{B}$.
Moreover, in such a design, it is known that $|\mathcal{B}|=\binom{v}{t}\cdot\lambda/\binom{k}{t}$. See  Khosrovshahi and Laue~\cite{tdesign} for more information about $t$-designs.
\end{definition}

\begin{definition}
Let $X$ be a set and let $\mathcal{B}$ be a family of subsets of $X$.
A $t\dash(v,k,\lambda)$ \defn{packing design} if $|X|=v$, $|B|=k$ for each $B\in\mathcal{B}$, and  each $t$-subset of $X$ is contained in \defn{at most} $\lambda$ sets in $\mathcal{B}$.
The \defn{packing number} $D_\lambda(v,k,t)$ is the number of blocks in a maximum $t\dash(v,k,\lambda)$ packing design.
Therefore $D_\lambda(v,k,t) \leq \binom{v}{t}\cdot\lambda/\binom{k}{t}$ with equality when a $t\dash(v,k,\lambda)$ design exists.
For more information on packing designs, see Mills and Mullin~\cite{millsmullin}.
\end{definition}

There is a connection between an $r\CBC$ and the complement of a packing design with appropriate parameters satisfying an additional property.

\begin{construction}
\label{con:design}
Let $g=m-(r+k)$ and $\mathcal{D}$ be a maximal $(g+1)\dash(m,g+2,k-1)$ packing design with vertex set $[m]$, block set $\mathcal{B}$, with the additional property (P):
\begin{quote}
 (P): no block in $\mathcal{B}$ appears more than $k-2$ times.
 \end{quote}
Observe that $|\mathcal{B}| \leq \frac{k-1}{r+k-1}\binom{m}{r+k-2}$.
Let $A$ be a matrix whose columns are the complements of the sets in $\mathcal{B}$, and thus each of the columns of $A$ has cardinality $r+k-2$.
\end{construction}

\begin{example}
\label{ex:fmkr3}
Let $r=1$, $k=3$, and $m=5$.
So $g=1$ and a maximal $2\dash(5,3,2)$ packing design has at most $6$ blocks.
See Figure \ref{fig:packingdesign} for a packing design $A$ which achieves this maximum.
Observe that this design is the complement of the $1\CBC(6,3,5)$ given in Example \ref{ex:fmkr1}.
\end{example}

\begin{figure}
\[\begin{bmatrix}
1&1&1&1&0&0\\
1&1&0&0&1&0\\
1&0&1&0&1&1\\
0&1&0&1&1&1\\
0&0&1&1&0&1\\
\end{bmatrix}\]
\caption{A maximal $2\dash(5,3,2)$ packing design.}
\label{fig:packingdesign}
\end{figure}

\begin{lemma}
\label{lemma:design_rcbc}
The matrix $A$ in Construction \ref{con:design} is an $r\CBC$.
Furthermore, the number of columns in $A$ is at most $F(k,m,r)$.
\end{lemma}

\begin{proof}
Suppose that $A$ has $n$ columns, let $1\leq c\leq k$, and let $J\subseteq[n]$ have cardinality $c$.
If $c\leq k-2$, then $\left |\bigcup_{j\in J}A_j \right| \geq r+k-2\geq r+c$.
If $c=k-1$ then, by property ($P$), not all the sets $A_j$ ($j\in J$) are equal, so $|\bigcup_{j\in J}A_j|\geq r+k-1$.
Therefore, to prove $A$ is an $r$\CBC, we need to show that if $|J|=k$, then $\left |\bigcup_{j\in J}A_j \right |\geq r+k$.

Assume that $A$ is not an $r\CBC$.
Then there exists $J\subseteq[n]$ of cardinality $k$ such that $\left |\bigcup_{j\in J}A_j \right|< r+k$.
Let $B_j\in\mathcal{B}$ ($j\in J$) be the complements of $A_j$.
Then $\left |\bigcap_{j\in J}B_j  \right| > m-(r+k)$, so $ \left |\bigcap_{j\in J}B_j \right| \geq g+1$.
So there exists a $(g+1)$-subset of $[m]$ which is contained in $k$ blocks of a $(g+1)\dash(m,g+2,k-1)$ packing design, which is a contradiction.
Therefore $\left |\bigcup_{j\in J}A_j\right |\geq r+k$ for every $k$-subset $J$ of~$[m]$.
Hence $A$ is an $r\CBC$, and since $|A_j|=r+k-2$ for each $j\in[n]$, we have that $n\leq F(k,m,r)$.
\end{proof}

In fact, since $\mathcal{B}$ is a maximal packing design, we have that $|\mathcal{B}|=F(k,m,r)$:

\begin{lemma}
Let $A$ be an $r\CBC$ which realizes $F(k,m,r)$ as given in Definition \ref{def:f}.
Then $\mathcal{B} = \{A_i^c\mid A_i\text{ is a column of }A\}$ is a $(g+1)$-$(m,g+2,k-1)$ packing design satisfying property (P).
Therefore a maximal $(g+1)$-$(m,g+2,k-1)$ packing design has $F(k,m,r)$ blocks.
\end{lemma}

\begin{proof}
First observe for each $A_i^c\in \mathcal{B}$ that $|A_i^c| = g+2$.
Let $X\subseteq [m]$ and $|X|=g+1$.
Then $|X^c|=r+k-1$.
Since $A$ is an $r\CBC$, there exist at most $k-1$ columns of $A$ which are contained by $X^c$.
Hence there are at most $k-1$ sets in $\mathcal{B}$ which contain $X$.
So $\mathcal{B}$ is a packing design.

Assume that $\mathcal{B}$ has $k-1$ identical sets.
Then $A$ contains $k-1$ identical columns.
This implies the existence of a set of $k-1$ columns whose union has cardinality $r+k-2$, however 
since $A$ is an $r\CBC(n,k,m)$, this union should have cardinality at least $r+k-1$, giving a contradiction.
So $\mathcal{B}$ must satisfy property (P).
It follows from Lemma \ref{lemma:design_rcbc} that a maximal $(g+1)$-$(m,g+2,k-1)$ packing design has $F(k,m,r)$ blocks.
\end{proof}

If the maximum packing design used in Construction \ref{con:design} is, in fact, a design, then equality is achieved in Lemma \ref{lem:fmkr_inequality}
and we can restate Theorem~\ref{thm:fmkr} in a simplified form, which generalizes a result of Bujt\'as and Tuza. 

\begin{corollary}
\label{cor:ifdesignexists}
Let $r\geq 0$, $k\geq 3$, $m\geq r+k$, $g=m-(r+k)$, and suppose there exists a $(g+1)\dash(m,g+2,k-1)$ design with property $(P)$.
If \[\frac{k-1}{r+k-1}\binom{m}{r+k-2}\leq n \leq (k-1)\binom{m}{r+k-1},\]
 then 
\[ N(n,k,m;r) = (r+k-1)n - \left\lfloor\dfrac{(k-1)\binom{m}{r+k-1}-n}{m-r-k+1}\right\rfloor.\]
\end{corollary}
Let $r=0$ and $\mathcal{B}$ be the set of all $(g+2)$-subsets of $[m]$.
Then $([m],\mathcal{B})$ is a $(g+1)$-$(m,g+2,k-1)$ design with property $(P)$ (in fact, the design is simple -- each block is distinct).
Therefore the hypotheses of Corollary~\ref{cor:ifdesignexists} are satisfied.
This shows that Theorem~1 of Bujt\'as and Tuza~\cite{tuza3} is a special case of Corollary~\ref{cor:ifdesignexists}.

\section{Achieving the trivial minimum}\label{sec:minimum}

We close with an inverse problem involving $r\CBC$s.

In Corollary \ref{cor:bounds}, we establish for all admissible parameters $n,k,m,r$ that $N(n,k,m;r) \geq (r+1)n$.
If this lower bound is achieved, then the associated matrix of such an $r\CBC$\ would contain exactly $r+1$ 1s in each column.
Hence, in a combinatorial batch code, each item must be stored on at least one server, and therefore $N(n,k,m;0)\geq n$.
In fact, the only combinatorial batch codes for which $N(n,k,m;0)=n$ are those for which $k=1$ (store each item on the same server) or $k\geq 2$ and $n\leq m$ (each item gets its own server).
However, when $r\geq 1$, this classification is nontrivial -- there exist parameters such that $k\geq 2$, $n>m$, and an $r\CBC(n,k,m)$ exists for which $N(n,k,m;r)=(r+1)n$.
%Recall that by Corollary~\ref{cor:bounds}, an $r\CBC(n,k,m)$ must have weight at least $(r+1)n$.
%When $r\geq 1$, it is possible for certain parameters with $k\geq 2$ and $n\geq m$ for $N(n,k,m;r)=(r+1)n$; that is each item is stored on exactly $r+1$ servers -- the minimum number of necessary servers per item.
See Figure \ref{fig:minCBC}~(a) for a $1\CBC(7,3,6)$ having weight $2n$.
In this section we find, for prescribed integers $m$, $k$, and $r$, the maximum number of items $n$ for which an erasure combinatorial batch code with redundancy $r$ exists that achieves this minimal weight.

\begin{figure}
\centering\begin{tabular}{cc}
$\begin{bmatrix}
1 & 0 & 0 & 0 & 0 & 1 & 1 \\
1 & 1 & 0 & 0 & 0 & 0 & 0 \\
0 & 1 & 1 & 0 & 0 & 0 & 0 \\
0 & 0 & 1 & 1 & 0 & 0 & 1 \\
0 & 0 & 0 & 1 & 1 & 0 & 0 \\
0 & 0 & 0 & 0 & 1 & 1 & 0 \\
\end{bmatrix}$
&
\raisebox{-.5\height}{\scalebox{.8}{\begin{tikzpicture}
\node[circle,draw] (1) at  (60:2) {$1$};
\node[circle,draw] (2) at  ( 0:2) {$2$};
\node[circle,draw] (3) at (300:2) {$3$};
\node[circle,draw] (4) at (240:2) {$4$};
\node[circle,draw] (5) at (180:2) {$5$};
\node[circle,draw] (6) at (120:2) {$6$};
\draw (1) -- (2) -- (3) -- (4) -- (5) -- (6) -- (1);
\draw (1) -- (4);
\end{tikzpicture}}}\\
\rule{0pt}{16pt}(a) & (b) 
\end{tabular}
\caption{(a) a $1$\CBC(7,3,6) achieving minimal weight $14$ and (b) its corresponding graph outlined in Theorem \ref{thm:girth}.}
\label{fig:minCBC}
\end{figure}
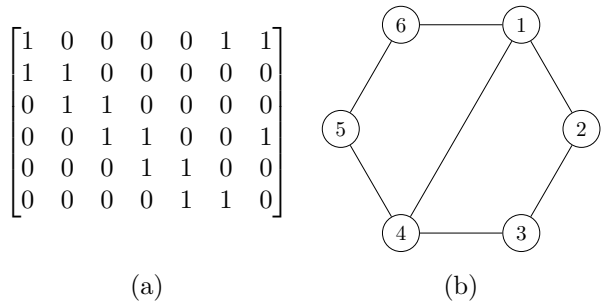

Given parameters $k$, $m$, and $r$, let $n(k,m;r)$ be the maximum value of $n$ such that $N(n,k,m;r)=(r+1)n$.
If there are no such $r\CBC$s for all $n\geq k$, we say $n(k,m;r)$ does not exist.
%In this section, we construct $r\CBC$s whose weight is $(r+1)n$ and identify $n(k,m;r)$ in some special cases.
Observe that if $k=1$, then any matrix in which each column has cardinality $r+1$ is sufficient, so $n(1,m;r)=\infty$.
We now consider when $k\geq 2$.

Suppose that $r=0$.
As noted before, a CBC$(n,k,m)$ has weight $n$ if and only if $m\geq n$ and each column is distinct.
So $n(k,m;0)$ exists (and equals $m$) if and only if $m\geq n$.

For larger $r$, this family of $r\CBC$s can be quite complex.
In what follows, we give a correspondence between $1\CBC$s with weight $2n$ and graphs on $m$ vertices with appropriate girth.

\begin{theorem}
\label{thm:girth}
Let $2 \leq k < m\leq n$.
Then $A$ is a $1\CBC(n,k,m)$ with weight $2n$ if and only if $A$ is the incidence matrix for a simple graph with $m$ vertices and girth at least $k+1$.
\end{theorem}

\begin{proof}
Let $G$ be a simple graph with girth at least $k+1$ and let $A$ be its incidence matrix.
Let $e_j$ denote the edge corresponding with column $A_j$.
Observe that $|A_j|=2$ for each $j\in[n]$ and hence $N(A)=2n$.

Let $J$ be a $c$-subset of $[n]$ with $c\leq k$ and $G'$ be the graph induced by the edge set $\{e_j\mid j\in J\}$.
Then $G'$ has no cycles, and therefore $G'$ is a forest.
Hence $G'$ has $a+c$ incident vertices, where $a\geq 1$ is the number of connected components of $G'$.
Therefore $\left |\bigcup_{j\in J}A_j\right|=a+c\geq 1+c$.
So $A$ is a $1\CBC(n,k,m)$.

Suppose $A$ is a $1\CBC(n,k,m)$ for which $N(A)=2n$.
Then $|A_j|=2$ for each $j\in[n]$, and therefore can be interpreted as the incidence matrix of a graph $G$.
If $G$ is not simple, then $G$ has two parallel edges, say which correspond to $A_p$ and $A_q$.
Then $|A_p\cup A_q| = 2 < 3$, which contradicts $A$ being a $1\CBC$.

Assume that $G$ does not have girth at least $k+1$.
Then there exists a cycle in $G$ with $c\leq k$ edges.
Let $J\subseteq[n]$ index the edges of the cycle.
So $\left |\bigcup_{j\in J}A_j\right | = c < 1+c$, a contradiction.
Therefore the graph $G$ has girth at least $k+1$.
\end{proof}

The graph and $1\CBC(7,3,6)$ in Figure \ref{fig:minCBC} correspond to each other.
Note that the graph has girth $4$.

\begin{corollary} For all $m \geq 1$, $n(2,m;1)=\binom m2$ and
$n(3,m;1)=\left\lfloor {m^2}/4\right\rfloor$.
\end{corollary}

\begin{proof}
A graph with $m$ vertices and girth at least 3 is a simple graph, meaning a complete graph maximizes the number of edges.
So $n(2,m;1)=\binom m2$.

A graph with $m$ vertices and girth at least 4 is a triangle-free graph.
By Tur\'{a}n's theorem, a triangle-free graph on $m$ vertices with a maximum number of edges is a complete bipartite graph whose parts are sizes $\left\lfloor m/2\right\rfloor$ and $\left\lceil m/2\right\rceil$.
So $n(3,m;1)=\left\lfloor m/2\right\rfloor\cdot\left\lceil m/2\right\rceil=\left\lfloor{m^2}/4\right\rfloor$.
\end{proof}

The maximum number of edges in a graph on $m$ vertices and girth $k+1\geq 4$ is known only for certain pairs of $m$ and $k$, but no other infinite families are known at this time~\cite{girth}.

\bibliographystyle{amsplain}
%\bibliography{mns-rcbc}

\providecommand{\bysame}{\leavevmode\hbox to3em{\hrulefill}\thinspace}
\providecommand{\MR}{\relax\ifhmode\unskip\space\fi MR }
\providecommand{\MRhref}[2]{%
  \href{http://www.ams.org/mathscinet-getitem?mr=#1}{#2}
}
\providecommand{\href}[2]{#2}

\end{document}